\documentclass[10pt,reqno]{amsart}
\usepackage{amsmath, amsthm, amssymb, mathtools, amsfonts}
\usepackage{url}

\usepackage{physics}
\usepackage{graphicx,lipsum}
\graphicspath{ {./downloads/} }
\usepackage[colorlinks,citecolor=red,urlcolor=blue,bookmarks=false,hypertexnames=true]{hyperref}
\usepackage{tikz}
\usetikzlibrary{calc}
\usetikzlibrary{shapes}
\usepackage[autostyle]{csquotes}
\makeatletter

\usepackage[colorlinks]{hyperref}
\hypersetup{
	colorlinks=false,
	citecolor=black,
	linkcolor=black,
	urlcolor=black}

\usepackage[nameinlink,capitalize]{cleveref}
\newtheorem{theorem}{Theorem}

\newtheorem{lemma}[theorem]{Lemma}

\newtheorem{remark}[theorem]{Remark}
\theoremstyle{definition}

\theoremstyle{remark}

\theoremstyle{conclusion}

\theoremstyle{observation}

\newcommand{\vol}[1]{\left\lvert#1\right\rvert}
\newcommand{\bvol}[1]{\big\lvert#1\big\rvert}
\newcommand{\Bvol}[1]{\Big\lvert#1\Big\rvert}

\newcommand{\inner}[2]{\left\langle#1, #2\right\rangle}
\newcommand{\conv}[0]{\mathrm{conv}}
\newcommand{\ball}[2]{B_{#1}^{#2}}
\newcommand{\soa}[1]{\{#1\}}
\newcommand{\RR}{\mathbb{R}}

\newcommand{\calS}{S}

\newcommand{\widebar}[1]{\mkern 1.5mu\overline{\mkern-1.5mu#1\mkern-1.5mu}\mkern 1.5mu}

\begin{document}
\title{A Generalization Of Gr\"unbaum's Inequality}

\author{Brayden Letwin}
\address{Brayden Letwin,  University of Alberta, Edmonton, Alberta, T6G 2G1, Canada}
\email{bletwin@ualberta.ca}

\author{Vladyslav Yaskin}
\address{Vladyslav Yaskin, Department of Mathematical \& Statistical Sciences, University of Alberta, Edmonton, Alberta, T6G 2G1, Canada}
\email{yaskin@ualberta.ca}

\thanks{The first author was supported by an NSERC USRA award. The second author was supported by an NSERC Discovery Grant.}


\begin{abstract}
    \noindent Grünbaum's inequality gives sharp  bounds between the volume of  a convex body and its part cut off by a hyperplane through the centroid of the body. We provide a generalization of this inequality for hyperplanes that do not necessarily contain the centroid. As an application, we obtain a sharp inequality that compares sections of a convex body to the maximal section parallel to it.
\end{abstract}
\maketitle

\section{Introduction}

    A \textit{convex body} $K$ is a compact convex subset of $\RR^n$ with non-empty interior. The \textit{centroid} (also called the \textit{center of mass}, or \textit{barycenter}) of $K$ is the point
    \begin{align*}
        g(K) = \frac{1}{\vol{K}}\int_K x \, dx.
    \end{align*}
    Here and throughout the paper, $|\cdot|$ denotes either the $n$-dimensional Lebesgue measure (volume) of a convex body or the  $(n-1)$-dimensional Lebesgue measure of its sections. It should be clear from the context what is meant in each particular case. An inequality of Grünbaum \cite{G} states if $K \subset \RR^n$ is a convex body with centroid at the origin then
    \begin{align}
        \left(\frac{n}{n+1}\right)^n \leq \frac{\vol{K \cap \xi^+}}{\vol{K}} \leq 1 - \left(\frac{n}{n+1}\right)^n, \quad \text{for all}\ \xi \in S^{n-1}. \label{eq1}
    \end{align}
    Here $\xi^+ = \{x \in \RR^n : \inner{x}{\xi} \geq 0 \}$ and   $\inner{\cdot}{\cdot}$ stands for the Euclidean inner product. We also write $S^{n-1}$ for  the unit sphere in $\RR^n$. The bounds in ($\ref{eq1}$) are sharp and equality occurs in the lower bound when, for example, $K$ is the cone
    \begin{align}
        K = \conv\left(\frac{-1}{n+1}\xi + \ball{2}{n-1}, \frac{n}{n+1}\xi\right) \label{eq2},
    \end{align}
    where we denote by $\ball{2}{n-1}$ the closed unit $(n-1)$-dimensional Euclidean ball in $\xi^\perp = \{x \in \RR^n : \inner{x}{\xi} = 0\}$.
     The upper bound in (\ref{eq1})  is easily obtained from the lower bound, and equality occurs in the upper bound when, for example, $K$ is the reflection about the origin of the cone in (\ref{eq2}).     
     For recent advancements in Grünbaum-type inequalities for sections and projections of convex bodies see \cite{FMY}, \cite{MNRY}, \cite{MSZ}, \cite{SZ}.

    In light of ($\ref{eq1}$), the goal of this paper is to establish a similar result with hyperplanes that do not necessarily contain the centroid. 
  For  $\alpha \in (-1, n)$ and $\xi \in S^{n-1}$ consider the half-space   
    \begin{align*}
    H_\alpha^+ = \{x \in \RR^n : \inner{x}{\xi} \geq \alpha h_K(-\xi)\},
    \end{align*}
    where $h_K$ is the support function for $K$ (see Section $2$ for the precise definition). We ask the following question: Are there positive constants $C_1(\alpha,n)$ and $C_2(\alpha,n)$, depending only on $\alpha$ and $n$, such that
    \begin{align}
        C_1(\alpha,n) \leq \frac{\vol{K \cap H_\alpha^+}}{\vol{K}} \leq C_2(\alpha,n),\label{eq3}
    \end{align}
    for every convex body $K\subset \mathbb R^n$ with centroid at the origin?
     We give an affirmative answer to this question.  Both bounds are sharp, and the  values of $C_1(\alpha, n)$ and $C_2(\alpha, n)$ are presented in Theorem $4$, which also discusses the equality cases. The case $n = 2$ for ($\ref{eq3}$) was obtained earlier in \cite{SY}, where it was used to prove a discrete version of Gr\"unbaum's inequality. It is important to note that in ($\ref{eq1}$) one bound automatically determines the other bound. On the other hand, the bounds in ($\ref{eq3}$) need to be shown separately.
    
    As an application of ($\ref{eq3}$) we obtain a generalization of the following result of Makai and Martini \cite{MM}; see also \cite{F}. Let $K \subset \RR^n$ be a convex body with centroid at the origin, then 
    \begin{align}
        \vol{K \cap \xi^\perp} \geq \left(\frac{n}{n+1}\right)^{n-1} \max_{t\in \RR} \Bvol{K \cap \left(\xi^\perp + t\xi\right)}, \quad \text{for all} \ \xi \in S^{n-1} \label{eq4}.
    \end{align}
    The bound is sharp, and equality holds again if, for example, $K$ is a cone as in ($\ref{eq2}$). In this paper, we establish an analogue of the inequality above for sections that do not necessarily pass through the centroid. Let $K \subset \RR^n$ be a convex body with centroid at the origin, $\alpha \in (-1, n)$, and $\xi \in S^{n-1}$. Consider the hyperplane
    \begin{align*}
        H_\alpha = \{x \in \RR^n : \inner{x}{\xi} = \alpha h_K(-\xi)\}.
    \end{align*}
    Then 
    \begin{align*}
        \vol{K \cap H_\alpha} \geq D(\alpha, n) \max_{t \in \mathbb{R}} \Bvol{K \cap \left(\xi^\perp + t\xi\right)},
    \end{align*}
    where $ D(\alpha, n)$ is a constant depending  only on $\alpha$ and $n$. The inequality is sharp, and the exact value of $D(\alpha, n)$ is discussed in Theorem $5$, along with equality cases.
    
\section{Preliminaries}

    The \textit{support function} $h_K : \RR^n \overset{}{\longrightarrow} \RR$ for a convex body $K \subset \RR^n$ is
    \begin{align*}
        h_K(\xi) = \max \soa{\inner{x}{\xi} : x \in K}.
    \end{align*}
    If $\xi \in S^{n-1}$ then $h_K(\xi)$ gives the signed distance from the origin to the supporting hyperplane for $K$ in the direction $\xi$. A result of Minkowski and Randon \cite[p.~58]{BF} states if $K \subset \RR^n$ is a convex body with centroid at the origin and $\xi \in S^{n-1}$, then
    \begin{align}
        \frac{1}{n}h_K(\xi) \leq h_K(-\xi) \leq nh_K(\xi). \label{eq5}
    \end{align}
    Note that the choice of bounds for $\alpha$ in Theorems $4$ and $5$ is a result of ($\ref{eq5}$).

    Let $\xi \in S^{n-1}$. The \textit{parallel section function} $A_{K,\xi}:\RR \overset{}{\longrightarrow} \RR$ for a convex body $K$ is
    \begin{align*}
        A_{K,\xi}(t) = \vol{K \cap \left(\xi^\perp + t\xi\right)}.
    \end{align*}
    
    \begin{lemma} Let $K \subset \RR^n$ be a convex body. Then $A_{K,\xi}^{1/(n-1)}$ is concave on its support, for every $\xi\in S^{n-1}$.
    \end{lemma} 
    \noindent For the proof of Lemma $1$, refer to \cite[p.~18]{Koldobsky}.
    
    Let $\xi \in S^{n-1}$. The \textit{volume cut-off function} $V_{K,\xi}:\RR \overset{}{\longrightarrow} \RR$ for a convex body $K \subset \RR^n$ is
    \begin{align*}
    V_{K,\xi}(t) = \int_t^\infty A_{K, \xi}(s) \, ds.
    \end{align*}
    The following result is also well-known, but we include a proof for completeness.

    \begin{lemma} Let $K \subset \RR^n$ be a convex body. Then $V^{1/n}_{K,\xi}$ is concave on its support, for every $\xi\in S^{n-1}$.
    \end{lemma}
    \begin{proof}
        Let $\lambda \in [0, 1]$ and $t_1, t_2 \in \mathrm{supp}(V_{K,\xi}) $. Note that
        \begin{multline*}
        \lambda \Big(K \cap \{x \in \RR^n : \inner{x}{\xi} \geq t_1\}\Big) + (1-\lambda)\Big(K \cap \{x \in \RR^n: \inner{x}{\xi} \geq t_2\}\Big) \\ \subset \Big(K \cap \{x \in \RR^n : \inner{x}{\xi} \geq \lambda t_1 + (1-\lambda) t_2\}\Big).
        \end{multline*}
        This, together with  the Brunn-Minkowski inequality (see \cite[p.~415]{Gar} or \cite[p.~369]{Sch}), implies that
        \begin{multline*}
        \Bvol{K \cap \{x \in \RR^n : \inner{x}{\xi} \geq \lambda t_1 + (1-\lambda) t_2\}}^{1/n} \\
        \hspace{0.83em} \geq \Bvol{\lambda \Big(K \cap \{x \in \RR^n : \langle x,\xi\rangle \geq t_1\}\Big) + (1-\lambda)\Big(K \cap \{x \in \RR^n : \langle x,\xi\rangle \geq t_2\}\Big)}^{1/n} \\
        \geq \lambda \Bvol{K \cap \{x \in \RR^n : \inner{x}{\xi} \geq t_1\}}^{1/n} + (1-\lambda)\Bvol{K \cap \{x \in \RR^n : \inner{x}{\xi} \geq t_2\}}^{1/n},
        \end{multline*}
        which proves the result.
    \end{proof}

    Let $K \subset \RR^n$ be a convex body and $\xi \in S^{n-1}$. The \textit{Schwarz symmetral} of $K$ with respect to $\xi$ is the convex body $\calS_{\xi}K$ such that for all $t \in [-h_K\left(-\xi \right), h_K\left(\xi \right)]$, the set $\calS_{\xi} K\cap \left(\xi^\perp + t\xi \right)$ is an $(n-1)$-dimensional Euclidean ball centered at $t\xi$ and $A_{K, \xi}(t) = A_{(\calS_\xi K), \xi}(t)$.
    By construction we obtain
    \begin{align}
        h_{K}(\pm \xi) = h_{\calS_\xi K}(\pm \xi) \quad \text{and} \quad V_{K, \xi}(t) = V_{(\calS_\xi K), \xi}(t) \label{eq6},
    \end{align}
    for all $t \in \mathbb{R}$. Note that the centroid of $\calS_\xi K$ lies on the line $\ell = \soa{t \xi : t \in \RR}$ due to the rotational symmetry of $\calS_\xi K$ about $\ell$. See \cite[p.~62]{Gar} for more information on Schwarz symmetrizations.

\section{Main Results}

    Before proving our main result, we will provide a simple remark that we will apply throughout the rest of the paper.
    \begin{remark}
        Let $K \subset \RR^n$ be a convex body with centroid at the origin. Let $\alpha \in (-1, n)$ and $\xi \in S^{n-1}$. Denote $\widebar K = K + h_K(-\xi)\xi$ and consider the two halfspaces $H_\alpha^+ = \{x\in\mathbb R^n : \langle x,\xi \rangle \geq \alpha h_K(-\xi)\}$ and $\widebar H_\alpha^+ = \soa{x \in \RR^n : \inner{x}{\xi} \geq (\alpha + 1)\inner{g(\widebar K)}{\xi}}$.
        Then
            $$\vol{K \cap H^+_\alpha} = \vol{\widebar K \cap \widebar H_\alpha^+}.$$
    \end{remark}
    \begin{proof}
    	Since $\widebar K$ is a translate of $K$, it is easy to see that the centroid of $\widebar K$ is translated by the same vector, i.e., $g(\widebar K)= h_K(-\xi)\xi$. Therefore,
    \begin{multline*}	\widebar H_\alpha^+  = \soa{x \in \RR^n : \inner{x}{\xi} \geq (\alpha + 1)\inner{g(\widebar K)}{\xi}}\\  = \soa{x \in \RR^n : \inner{x}{\xi} \geq (\alpha + 1)h_K(-\xi)}= H_\alpha^+ + h_K(-\xi)\xi,
    \end{multline*}
        and the result follows.
    \end{proof}
    Statements analogous  to Remark $3$ also hold when $\geq$ is replaced with $\leq$ or $=$. We will now prove our main result.
    
    \begin{theorem}
        Let $K \subset \RR^n$ be a convex body with centroid at the origin. Let $\alpha \in (-1, n)$ and $\xi \in S^{n-1}$. Consider the half-space
        \begin{align*}
        H_\alpha^+ = \{x\in\mathbb R^n : \langle x,\xi \rangle \geq \alpha h_K(-\xi)\}.
        \end{align*}
        Then 
        \begin{align*}
        C_1(\alpha, n) \leq \frac{\left\lvert K \cap H_\alpha^+ \right\rvert}{\left\lvert K \right\rvert} \leq C_2(\alpha, n). 
        \end{align*}
        where 
        \begin{align*}
        C_1(\alpha, n) = \begin{cases} 
              \left(\frac{n-\alpha}{n+1}\right)^n, & \mbox{ if } \alpha \in (-1, 0], \\
              \left(\frac{n}{n+1}\right)^n(\alpha+1)^{n-1}(1-\alpha n) , & \mbox{ if } \alpha \in (0, 1/n), \\
              0 , & \mbox{ if } \alpha \in [1/n, n),
           \end{cases}
        \end{align*}
           and 
        \begin{align*}
        C_2(\alpha, n) = \begin{cases} 
              1 - \left(\frac{n(\alpha + 1)}{n+1}\right)^n , & \mbox{ if } \alpha \in (-1, 0], \\
              c(\alpha, n) , & \mbox{ if } \alpha \in (0, n). \\
           \end{cases}
        \end{align*}
        $c(\alpha, n)$ is a constant depending only on $\alpha$ and $n$. Determining the explicit value of $c(\alpha, n)$ involves finding the roots of a high-degree rational function. The lower bounds and upper bounds are sharp, and equality cases are discussed in the proof below.
    \end{theorem}

    \begin{proof}
    Given $K$ as written above, consider the Schwarz symmetral $\calS_\xi K$. Using the observations in ($\ref{eq6}$) and Fubini's theorem we can conclude that the centroid of $\calS_\xi K$ is at the origin and that $|K \cap H_\alpha^+| = |(\calS_\xi K) \cap H_\alpha^+|$ for all $\alpha \in (-1, n)$. Therefore we will prove the result with $\calS_\xi K$, which we will denote by $K$ for brevity. By Remark $3$, it suffices to find bounds for $|\widebar K \cap \widebar H_\alpha^+|$, and after further abuse of notation, we will write $K$ for $\widebar K$ and $H_\alpha^+$ for $\widebar H_\alpha^+$. We will also write $H_\alpha = \soa{x \in \RR^n : \inner{x}{\xi} = (\alpha + 1)\inner{g(K)}{\xi}}$ and $H_\alpha^- = \soa{x \in \RR^n : \inner{x}{\xi} \leq (\alpha + 1)\inner{g(K)}{\xi}}$. Let us remark that $\xi^\perp$ is now a supporting hyperplane of $K$ and $0 \in \partial K$.

    Let us first consider the case $\alpha \in (-1, 0]$. We will obtain the upper bound. Observe that
    \begin{align*}
        \vol{K \cap H_\alpha^-} = \vol{K} - \vol{K \cap H_\alpha^+}.
    \end{align*}
    Denote by $K/(\alpha + 1)$ the dilation of $K$ by a factor of $1/(\alpha + 1) > 1$, and also write $H_\alpha^-/(\alpha + 1) = \soa{x \in \RR^n : \inner{x}{\xi} \leq \inner{g(K)}{\xi}}$. Since $0 \in K$, we obtain $K \subset K/(\alpha + 1)$ and thus
    \begin{multline*}
    \vol{K \cap H_\alpha^-} = (\alpha+1)^n \vol{\frac{1}{\alpha+1}K \cap \frac{1}{\alpha+1}H_\alpha^-} \\ 
    \hspace{-11.8em} \geq (\alpha+1)^n \Bvol{K \cap \{x \in \mathbb{R}^n : \langle x,\xi\rangle \leq \inner{g(K)}{\xi} \}}\\ = (\alpha + 1)^n \Bvol{\big(K - g(K)\big)\cap \{x \in \mathbb{R}^n : \langle x,\xi\rangle \leq 0 \}} \geq (\alpha + 1)^n\left(\frac{n}{n+1}\right)^n \left\lvert K \right\rvert,
    \end{multline*}
    where we used Gr\"unbaum's inequality  ($\ref{eq1}$). Therefore, for $\alpha \in (-1, 0]$, we have
    \begin{align*}
    \frac{\left\lvert K \cap H_\alpha^+\right\rvert}{\left\lvert K \right\rvert} \leq 1 - \left(\frac{n(\alpha + 1)}{n+1}\right)^n,
    \end{align*}
    as desired. 
    
    We will now obtain the lower bound for $\alpha \in (-1, 0]$. By Lemma $2$, $V^{1/n}_{K,\xi}$ is concave on its support. Hence,
    \begin{align*}
    \left\lvert K \cap H_\alpha^+\right\rvert^{1/n}
    &= {V_{K,\xi}}^{{1/n}}\big((\alpha + 1)\inner{g(K)}{\xi}\big)= {V_{K,\xi}}^{{1/n}}\big(- \alpha \cdot 0 + (\alpha + 1)\inner{g(K)}{\xi}\big)  \\
    &\geq -\alpha{V_{K,\xi}}^{{1/n}}(0) + (\alpha + 1){V_{K,\xi}}^{{1/n}}\big(\inner{g(K)}{\xi}\big).
    \end{align*}
    Using Gr\"unbaum's inequality and the observation that $V_{K,\xi} (0) = \left\lvert K \right\rvert$, we have
    $$
    \left\lvert K \cap H_\alpha^+ \right\rvert^{1/n}\geq -\alpha \left\lvert K \right\rvert^{1/n} + (\alpha + 1)\left(\frac{n}{n+1}\right) \left\lvert K \right\rvert^{1/n},
    $$
    which implies for $\alpha \in (-1, 0]:$
    \begin{align*}
    \left(\frac{n-\alpha}{n+1}\right)^n \leq \frac{\left\lvert K \cap H_\alpha^+ \right\rvert}{\left\lvert K \right\rvert}.
    \end{align*}
    Thus, we have shown the bounds for $\alpha \in (-1, 0]$.

    We will now investigate the case $\alpha \in (0, n)$. We will prove the upper bound first. We can assume that $ K \cap H_\alpha$ is non-empty, otherwise the bound is trivial. Let $\ball{2}{n-1}$ be the unit $(n-1)$-dimensional Euclidean ball in $\xi^\perp$. By continuity we can find $r_1 \geq 0$ such that 
    \begin{align*}
    K \cap \xi^\perp \subset r_1 \ball{2}{n-1} \quad \text{and}\quad \bvol{\conv(r_1\ball{2}{n-1}, K \cap H_\alpha)} = \vol{K \cap H_\alpha^{-}}.
    \end{align*}
    Denote $L^- = \conv(r_1\ball{2}{n-1}, K \cap H_\alpha)$. Then again by continuity, there are $r_2 \geq 0$ and $\mu$ with $(\alpha + 1)\inner{g(K)}{\xi} < \mu < h_K(\xi)$ such that
    \begin{align*}
        \bvol{\conv(K \cap H_\alpha, r_2\ball{2}{n-1}+ \mu \xi)} = \vol{K \cap H_\alpha^+}
    \end{align*}
    and
    \begin{align*}
    L^- \cup \conv( K \cap H_\alpha,r_2\ball{2}{n-1}+\mu \xi) = \conv(r_1\ball{2}{n-1}, r_2\ball{2}{n-1}+\mu\xi).
    \end{align*}
    Denote $L^+ = \conv (K \cap H_\alpha, r_2\ball{2}{n-1}+ \mu \xi)$. Then $L = L^- \cup L^+$ is a truncated cone whose sections parallel to $\xi^\perp$ are Euclidean balls; see Figure \ref{fig1}.
    \begin{figure}[h]
    \centering
    \includegraphics[scale=0.54]{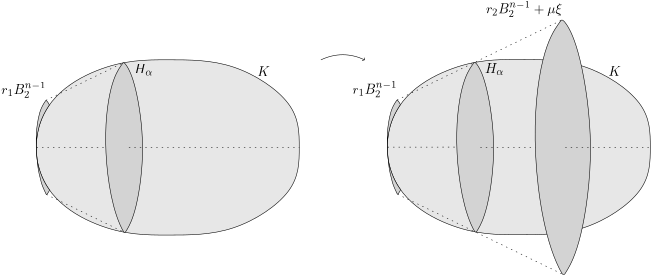}
    \caption{Constructing $r_1\ball{2}{n-1}$ and $r_2\ball{2}{n-1}+\mu \xi$.} \label{fig1}
    \end{figure}
    Note that
    $\inner{g(L^-)}{\xi} \leq \inner{g(K\cap H_\alpha^-)}{\xi}$ and $\inner{g(L^+)}{\xi} \leq \inner{g(K\cap H_\alpha^+)}{\xi}$, and thus
    \begin{align*}
    \inner{g(L)}{\xi} \leq \inner{g(K)}{\xi}.
    \end{align*}
    By construction, we have $\vol{K} = \vol{L}$ and
    \begin{multline*}
        \left\lvert K \cap \{x \in \mathbb{R}^n : \inner{x}{\xi} \geq (\alpha + 1)\langle g(K), \xi\rangle\} \right\rvert  \\
    \hspace{4.4em}\qquad \qquad \qquad \qquad\qquad=\left\lvert L \cap \{x \in \mathbb{R}^n : \inner{x}{\xi} \geq (\alpha + 1)\langle  g(K), \xi \rangle\} \right\rvert \\
    \qquad \qquad \qquad \qquad \qquad\leq \left\lvert L \cap \{x \in \mathbb{R}^n : \inner{x}{\xi} \geq (\alpha + 1)\langle g(L), \xi\rangle\}\right\rvert.
    \end{multline*}
    Hence, it suffices to work with $L$ instead of $K$. After rescaling, we may assume that $h_{L}(\xi) = 1$ and then for $0 \leq t \leq 1$ we can assume that
    \begin{align}
        A_{L, \xi}(t) = (mt + b)^{n-1}, \label{eq9}
    \end{align}
    where either $m = 0$ and $b > 0$, or $b \geq 0$ and either ($1$) $m > 0$ or ($2$) $m < 0$ and $m + b \geq 0$. We will focus on the case $m \neq 0$ first, and address the case $m = 0$ later. By Fubini's theorem and ($\ref{eq9}$) we obtain
    \begin{align*}
        \left\lvert L \right\rvert  = \int_{0}^1 A_{L, \xi}(t)\, dt  =  \frac{(b+m)^n - b^n}{mn},
    \end{align*}
    and similarly we find
    \begin{align*}
    \inner{g(L)}{\xi} = \frac{1}{\vol{L}}\int_{0}^1 t A_{L, \xi}(t) \, dt = \frac{b^{n+1} + (mn - b)(b+m)^n}{m(n+1)\big( (b+m )^n -b^n\big)}.
    \end{align*}
    Denote $G_L = (\alpha + 1)\inner{g(L)}{\xi}$. Now we can compute
    \begin{align*}
        \frac{\vol{L \cap H_\alpha^+}}{\vol{L}} = \frac{1}{\vol{L}}\int_{G_L}^1 A_{L, \xi}(t) \, dt =\frac{(b+m)^n - (b+mG_L)^n}{(b+m)^n - b^n}.
    \end{align*}
    Denote by $\varphi$ the above equation of $m$ and $b$ for when $m \neq 0$. If $b > 0$ then we have
    $\varphi(m, b) \xrightarrow{\: m \to 0 \:} (1 - \alpha)/2$, which is readily verified to agree with the case $m = 0$. Making the change of variables $z = b/m$ allows us to write $\varphi$ as a function of $z$. That is, we obtain
    \begin{align*}
        \varphi(z) =\frac{(z+1)^n - (z+G_L)^n}{(z+1)^n - z^n},
    \end{align*}
    where
    \begin{align*}
    G_L = (\alpha + 1)\frac{z^{n+1} + (n - z)(z + 1)^n}{(n+1)\big( (z + 1)^n -z^n\big)},
    \end{align*}
    for $z \in (-\infty, -1] \cup [0, \infty)$. For $\alpha \in (0, n)$, the rational function $\varphi$ is not monotonic, so determining $c(\alpha, n)$ becomes an unfeasible task, as this involves finding roots of high-degree rational functions. When $n = 2$ one can explicitly solve for $c(\alpha, n)$ (see \cite{SY} for the derivation):
    \begin{align*}
    c(\alpha, 2) = 
       \begin{cases} 
           \frac{5-3\alpha}{9\left(\alpha + 1\right)},& \mbox{ if } \alpha \in (0, 1), \\
           \frac{1}{9}(2-\alpha)^2, &  \mbox{ if } \alpha \in [1, 2). \\
        \end{cases}
    \end{align*}
    When $n\ge 3$ and $\alpha \in (0, n)$ all we can write is
    \begin{align*}
        \frac{\vol{K \cap H_\alpha^+}}{\vol{K}} \leq c(\alpha, n) = \sup\{\varphi(z): z \in (-\infty, -1] \cup [0, \infty)\}.
    \end{align*}

    We will now obtain the lower bound for $\alpha \in (0, n)$. Note for $\alpha \in [1/n, n)$ if $K$ is the cone
    \begin{align*}
        K = \conv\left(-\frac{n}{n+1}\xi,\frac{1}{n+1}\xi + \ball{2}{n-1}\right),
    \end{align*}
    then $\vol{K \cap H_\alpha^+} = 0$. Therefore we cannot do better than $C_1(\alpha, n) = 0$. 
    
    Now assume that $\alpha \in (0, 1/n)$. By continuity, there is $v \geq h_K(\xi)$ such that
    \begin{align*}
        \bvol{\conv(K \cap H_\alpha, v\xi)} = \vol{K \cap H_\alpha^+}.
    \end{align*}
    Denote $M^+ = \mathrm{conv}(K \cap H_\alpha, v\xi)$. Then again by continuity there are $r$ and $\beta$ with $r > 0$ and $0 \le \beta \le (\alpha + 1)\inner{g(K)}{\xi}$ such that
    \begin{align*}
    \mathrm{conv}(r\ball{2}{n-1} + \beta\xi, M^+) = \mathrm{conv}(r\ball{2}{n-1} + \beta\xi, v\xi),
    \end{align*}
    and 
    \begin{align*}
    \bvol{\conv(0, r \ball{2}{n-1} + \beta\xi, K \cap H_\alpha)} = \vol{K\cap H_\alpha^-}.
    \end{align*}
    Denote $M^- = \conv(0, r \ball{2}{n-1} + \beta\xi, K \cap H_\alpha)$. If $\beta>0$ then $M = M^- \cup M^+$ is a convex body formed by the union of two cones with a common base in $\xi^\perp + \beta\xi$; see Figure~\ref{fig2}. Such a body will be called a double cone. If $\beta =0$ then $M$ is a cone.
    \begin{figure}[h]
    \centering
    \includegraphics[scale=0.37]{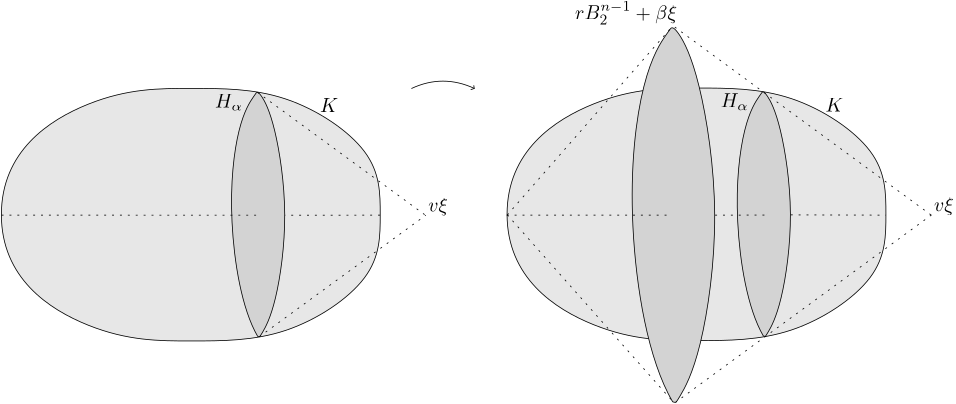}
    \caption{Constructing $r\ball{2}{n-1} + \beta\xi$ and $v\xi$.}\label{fig2}
    \end{figure}
    
    \noindent Note that
    $\inner{g(M^-)}{\xi} \geq \inner{g(K\cap H_\alpha^-)}{\xi}$ and $\inner{g(M^+)}{\xi} \geq \inner{g(K\cap H_\alpha^+)}{\xi}$, and thus
    \begin{align*}
    \inner{g(M)}{\xi} \geq \inner{g(K)}{\xi}.
    \end{align*}
    As a result we have constructed a convex body $M$ where $\left\lvert K \right\rvert = \left\lvert M \right\rvert$ and
    \begin{align*}
    & \left\lvert K \cap \{x \in \mathbb{R}^n : x_1 \geq (\alpha + 1)\langle g(K), \xi\rangle\} \right\rvert  \\
    &\qquad \qquad \qquad \qquad\qquad=\left\lvert M \cap \{x \in \mathbb{R}^n : x_1 \geq (\alpha + 1)\langle  g(K), \xi\rangle\} \right\rvert \\
    &\qquad \qquad \qquad \qquad \qquad\geq \left\lvert M \cap \{x \in \mathbb{R}^n : x_1 \geq (\alpha + 1)\langle g(M), \xi\rangle\}\right\rvert.
    \end{align*}
    Hence, it suffices to work with $M$ instead of $K$. After rescaling, we may assume that $h_{M}(\xi) = 1$ and $\vol{r \ball{2}{n-1}} = n.$ Define \begin{align*}
M_1 = M \cap \{x \in \mathbb{R}^n \mid \langle x,\xi \rangle \leq \beta\} \quad \text{and} \quad M_2 = M \cap \{x \in \mathbb{R}^n \mid \langle x,\xi  \rangle \geq \beta\}, 
\end{align*}
to be the cones forming $M$. $M_1$ is degenerate if $\beta =0$. Since the heights of $M_1$ and $M_2$ are $ \beta$ and $1-\beta$ respectively, and $\vol{r\ball{2}{n-1}} = n$, we get $\left\lvert M_1 \right\rvert =\beta$ and $\left\lvert M_2 \right\rvert=1-\beta$. 
    It is a well-known fact that the centroid of a cone in $\mathbb R^n$ divides its height in the ratio $[1:n]$. Hence, we obtain $\langle g(M_1), \xi\rangle = (\beta n)/(n+1) $ and $\langle g(M_2), \xi\rangle = (\beta n+1)/(n+1)$, and thus it follows that
    \begin{multline*}
    \langle g(M), \xi\rangle = \left\lvert M_1 \right\rvert \langle g(M_1), \xi\rangle+\left\lvert M_2 \right\rvert\langle g(M_2), \xi\rangle\\
    =\beta \frac{\beta n}{n+1}+(1-\beta) \frac{\beta n+1}{n+1}= \frac{\beta(n-1) + 1}{n+1}.
    \end{multline*}
    Denote $G_M = (\alpha + 1)\langle g(M), \xi\rangle$. We are interested in computing the volume of the intersection of $M$ with the halfspace $H_\alpha^+=\{x \in \mathbb R^n : \inner{x}{\xi} \geq G_M\}$. We will consider two cases, first when $0 \le \beta \leq G_M$, and then when $G_M \leq \beta < 1.$ These cases are equivalent to $0 \le \beta \leq \frac{\alpha+1}{2-(n-1)\alpha}$ and $\frac{\alpha+1}{2-(n-1)\alpha} \leq \beta < 1$, respectively. In the first case, note that $M\cap H_\alpha^+$ is a cone homothetic to $M_2$ with the homothety coefficient equal to $(1-G_M)/(1-\beta).$ Therefore,
    \begin{align*}
    \left\lvert M \cap H_\alpha^+ \right\rvert & =\left(\frac{1-G_M}{1-\beta }\right)^n (1-\beta)
    \nonumber = \frac{(1-G_M)^n}{(1-\beta)^{n-1}}.
    \end{align*}
    In the second case, $M \cap H_\alpha^-$ is a cone homothetic to $M_1$ with the homothety coefficient equal to $G_M / \beta.$ Thus,
    \begin{align*}
    \left\lvert M \cap H_\alpha^+ \right\rvert &=1-\left\lvert M \cap H_\alpha^-\right\rvert= 1- \left(\frac{G_M}{\beta} \right)^n \beta = 1- \frac{G_M^n}{\beta^{n-1}}.
    \end{align*}
    Summarizing, $\left\lvert M \cap H_\alpha^+ \right\rvert$ is equal to the following piecewise function
     $$\psi(\beta) = \begin{cases}
     \frac{\left(1-G_M\right)^n}{(1-\beta)^{n-1}}, & \mbox{ if } 0 \le \beta \leq \frac{\alpha+1}{2-(n-1)\alpha},\\
       1- \frac{G_M^n}{\beta^{n-1}}  , & \mbox{ if } \frac{\alpha+1}{2-(n-1)\alpha} \leq \beta < 1. \end{cases}
     $$
         Our goal is to find the infimum of $\psi$ on $[0,1)$ when $\alpha \in (0, 1/n)$. Calculations show that the derivative of $\psi$ vanishes at $\beta_0= \big((n+1)\alpha\big)/(\alpha + 1)\in (0, \frac{\alpha+1}{2-(n-1)\alpha})$. Furthermore, $\psi$ is decreasing on $[0,\beta_0)$ and increasing on $(\beta_0, 1)$. Thus, the minimum of $\psi$ is
    $$
     \psi(\beta_0)  = \left(\frac{n}{n+1}\right)^n(\alpha+1)^{n-1}(1-\alpha n),
    $$
    which is the value of $C_1(\alpha,n)$ when $\alpha \in (0, 1/n)$.

    We will now discuss the equality cases. Recall that in both the upper   and lower bound constructions, we performed   the Schwarz symmetrization to transform the sections of $K$ in the direction of $\xi$ into $(n-1)$-dimensional Euclidean balls. We also performed dilations and translations. If we have an equality body $K$ for either bound under these operations, then we can undo these operations to produce a new body whose sections are no longer $(n-1)$-dimensional Euclidean balls but instead $(n-1)$-dimensional convex bodies homothetic to each other.

    We will start classifying equality cases for the upper bound. For $\alpha \in (-1, 0]$, we have equality from the equality conditions of Gr\"unbaum's theorem, in other words $L = \mathrm{conv}(B, v)$ is a cone with its base $B$ being an $(n-1)$-dimensional convex body lying parallel to $\xi^\perp$ in $\xi^+$ and vertex $v$ lying in $\xi^- = \{x \in \mathbb{R}^n : \langle x,\xi \rangle \leq 0 \}$. For $\alpha \in (0, n)$, $L$ is  the convex hull of an $(n-1)$-dimensional convex body $B$ and a homothetic copy of $B$. The coefficient of homothety is not explicit, but it can be found numerically for each $n$ and $\alpha$. In particular, up to translation and dilation, $L$ is the convex hull of 
an $(n-1)$-dimensional convex body $B$ in  $\xi^\perp$ and its homothetic copy $\lambda B$ in $\xi^\perp+\xi$, where $\lambda = 1+\frac{1}{z_0}$ and $z_0$ is the point where the maximum of $\varphi$ is attained.
    \begin{figure}[h]
    \centering
    \includegraphics[scale=0.23]{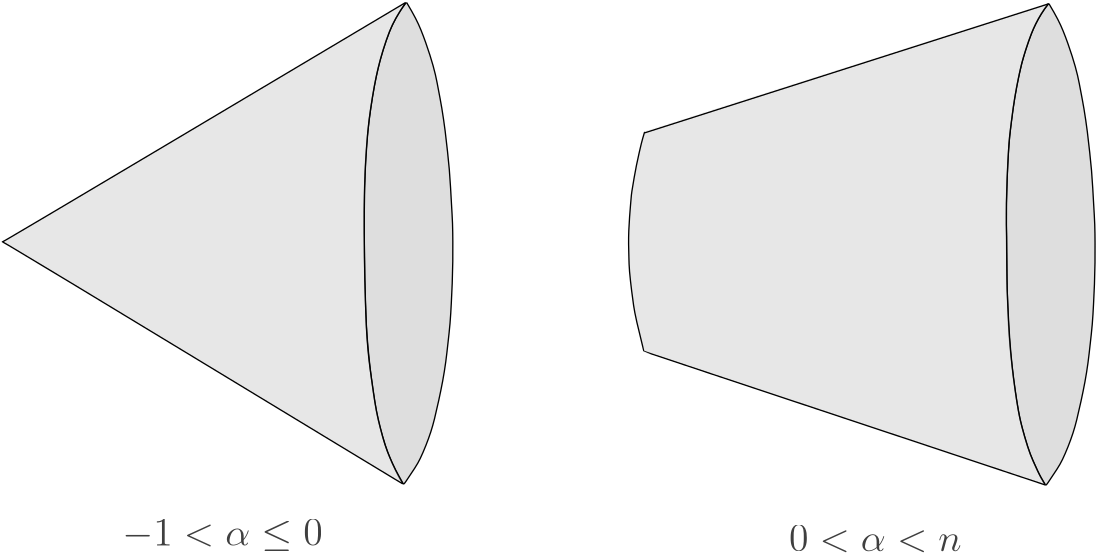}
    \caption{Extremizing shapes for the upper bound.}\label{fig3}
\end{figure}
    
    We will now classify equality cases for the lower bound. For $\alpha \in (-1, 0]$, we have equality from the equality conditions of Gr\"unbaum's theorem, in other words $M = \mathrm{conv}(B, v)$ is a cone with its base $B$ being an $(n-1)$-dimensional convex body lying parallel to $\xi^\perp$ in $\xi^-$ and vertex $v$ lying in $\xi^+$. For $\alpha \in (0, 1/n)$, the extremizing body is the union of two cones which share the same base, and whose heights are proportional to each other with the coefficient  $\beta_0/(1-\beta_0)$.  Recall that $\beta_0$ is the unique point of minimum for the function $\psi$. As $\alpha$ increases from $0$ towards $1/n$, $\beta_0$ increases from $0$ to $1$, so $B$ shifts in the direction of $\xi$. When $\alpha =1/n$ we have the equality $|K\cap H_\alpha^+|=0$ only in the case when $h_K(\xi)= \frac1n h_K(-\xi)$ (assuming the centroid of $K$ is at the origin). The latter is possible only  when $K$ is a cone, which follows from the equality case in  \eqref{eq5}.   For $\alpha \in ( 1/n, n)$,  we have  many bodies centered at the origin with property $h_K(\xi)< \alpha h_K(-\xi)$. All of them satisfy $|K\cap H_\alpha^+|=0$. 
    
    \newpage
    
    \end{proof}
    \begin{figure}[h]
    \centering
    \includegraphics[scale=0.26]{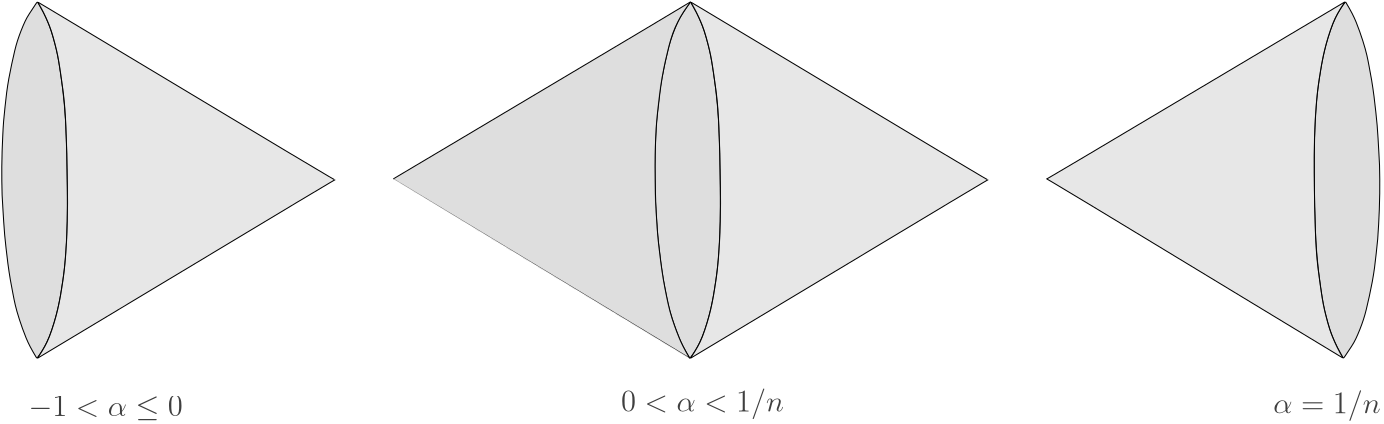}
    \caption{Extremizing shapes for the lower bound.}\label{fig3}
    \end{figure}
    As an application of Theorem $4$ we obtain a generalization of the result of Makai and Martini \cite{MM} stated in the introduction.
    \begin{theorem}
    Let $K $ be a convex body with centroid at the origin. Let $\xi \in S^{n-1}$ and $\alpha \in (-1, n)$. Consider the hyperplane
    \begin{align*}
    H_\alpha = \{x\in\mathbb R^n : \langle x,\xi \rangle = \alpha h_K(-\xi)\}.
    \end{align*}
    Then 
    \begin{align*}
    \left \lvert K \cap H_\alpha\right \rvert \geq D(\alpha, n)\sup_{t \in \mathbb{R}}\left \lvert K \cap \big(\xi^\perp + t\xi\big)\right \rvert,
    \end{align*}
    where 
    \begin{align*}
    D(\alpha, n) =  \begin{cases}
          \left(\frac{n(\alpha+1)}{n+1}\right)^{n-1} , &\mbox{ if } \alpha \in (-1, 0], \\
         \left(\frac{n-\alpha}{n+1}\right)^{n-1} , &\mbox{ if } \alpha \in (0, 1/n], \\
          0 , &\mbox{ if } \alpha \in (1/n, n).\\
       \end{cases}
    \end{align*}
    The bound is sharp and equality cases are discussed in the proof below.
    \end{theorem}

    \begin{proof}
    Note for $\alpha \in (1/n, n)$, if $K$ is the cone
    \begin{align}\label{cone}
        K = \conv\left(\frac{-n}{n+1}\xi , \frac{1}{n+1}\xi+ \ball{2}{n-1}\right),
    \end{align}
    then it follows that $\vol{K \cap H_\alpha} = 0$. Therefore for such $\alpha$ we cannot do better than $D(\alpha, n) = 0$. 
    
    We will now consider $\alpha \in (-1, 0]$.		
    We can assume that
    \begin{align*}
      \left \lvert K \cap H_\alpha \right \rvert < \sup_{t \in \mathbb{R}}\left \lvert K \cap \big(\xi^\perp + t\xi\big) \right \rvert,
    \end{align*}
    otherwise, the theorem follows immediately.     
    We will apply the Schwarz symmetrization $\calS_\xi$ to $K$. Abusing notation, we will denote the new body again by $K$. We will write $$t_0 = \min\{t \in \mathbb{R} : A_{K, \xi}(t) = \max_{t \in \mathbb{R}}A_{K, \xi}(t)\},$$ so that $K \cap (\xi^\perp + t_0\xi)$ is a section of $K$ orthogonal to $\xi$ of maximal volume. Since $0 < \left \lvert K \cap H_\alpha \right \rvert < \left \lvert K \cap \big(\xi^\perp + t_0\xi\big) \right \rvert$ we can find a cone with base equal to $K \cap \big(\xi^\perp + t_0\xi\big)$ and section equal to $K \cap H_\alpha$. Such a cone is uniquely determined by these two sections. Denote this cone by $N_1$. Let $\gamma \xi$ be the vertex of $N_1$, for some number $\gamma$ (either positive or negative).  Due to the convexity of $K$, $\gamma\xi$ lies outside of $K$ or on the boundary of $K$. Define $N_2$ to be the cone with base equal to $K \cap H_\alpha$ and vertex $\gamma\xi$; see Figure \ref{fig4}. Finally, we will let $H_\alpha^*$ be the halfspace bounded by the hyperplane $H_\alpha$ that contains $N_2$.
    \begin{figure}[h]
    \centering
    \includegraphics[scale=0.4]{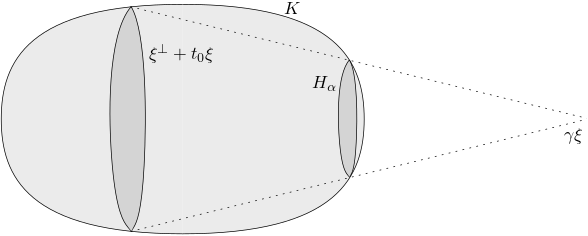}
    \caption{Constructing $N_1$ and $N_2$.}\label{fig4}
    \end{figure}
    We will consider two cases: $H_\alpha^* = H_\alpha^+ = \{x \in \mathbb{R}^n : \langle x, \xi \rangle \geq \alpha h_K(-\xi) \}$ and $H_\alpha^* = H_\alpha^- = \{x \in \mathbb{R}^n : \langle x, \xi \rangle \leq \alpha h_K(-\xi) \} $.  Denote $h = \alpha h_K(-\xi) $. When $H_\alpha^* = H_\alpha^+$ the following inequality holds:
    \begin{align*}
    &\left \lvert K \cap H_\alpha \right \rvert = \frac{\left \lvert N_2 \right \rvert n}{{|\gamma - h|}} \geq \frac{\left \lvert K \cap H^+_\alpha \right \rvert n}{ {|\gamma - h|}}.
    \end{align*}
    \noindent Then by Theorem $4$ and using that $$\vol{K} = \left \lvert K \cap H_\alpha^+ \right \rvert + \left \lvert K \cap H_\alpha^- \right \rvert \geq C_1(\alpha, n)\vol{K} + \vol{K \cap H_\alpha^-}$$ we note that $\big(1 - C_1(\alpha, n)\big)|K| \geq |K \cap H_\alpha^-|$. We arrive at the following estimates
    \begin{align*}
    \vol{K \cap H_\alpha} \geq \frac{\left \lvert K \cap H^+_\alpha \right \rvert n}{{|\gamma - h|}}& \geq C_1(\alpha, n)\frac{\left \lvert K \right \rvert n}{{|\gamma - h|}} \\
    &\geq \frac{C_1(\alpha, n)}{1 - C_1(\alpha, n)}\frac{\left \lvert K \cap H_\alpha^-\right \rvert n}{ {|\gamma - h|}} \geq \frac{C_1(\alpha, n)}{1 - C_1(\alpha, n)}\frac{\left \lvert N_1 \setminus N_2 \right \rvert n}{ {|\gamma - h|}}.
    \end{align*}
    Expressing the volumes of $N_1$ and $N_2$ in terms of their bases, we see
    \begin{align*}
    |K \cap H_\alpha| \geq\ & \frac{C_1(\alpha, n)}{1 - C_1(\alpha, n)}\frac{(\left \lvert N_1 \right \rvert - \left \lvert N_2 \right \rvert) n}{ {|\gamma - h|}}\\
    =\ & \frac{C_1(\alpha, n)}{1 -C_1(\alpha, n)}\frac{\left \lvert K \cap \big(\xi^\perp + t_0\xi\big) \right \rvert  {|\gamma - t_0|}}{ {|\gamma - h|}} - \frac{C_1(\alpha, n)}{1 - C_1(\alpha, n)}\left \lvert K \cap H_\alpha \right \rvert.
    \end{align*}
    And so,
    \begin{multline*}
       |K \cap H_\alpha| \geq  \frac{\frac{C_1(\alpha, n)}{1 - C_1(\alpha, n)}}{1 + \frac{C_1(\alpha, n)}{1 - C_1(\alpha, n)}}\frac{|K \cap (\xi^\perp + t_0 \xi) ||\gamma - t_0|}{|\gamma - h|} \\
       = C_1(\alpha, n)\frac{|K \cap (\xi^\perp + t_0 \xi) ||\gamma - t_0|}{|\gamma - h|}.
    \end{multline*}
    Because $N_1$ is a homothetic copy of $N_2$, we can write $$ \frac{|\gamma - t_0|}{|\gamma - h|}= \frac{\left \lvert K \cap \big(\xi^\perp + t_0\xi\big) \right \rvert^{{1/(n-1)}}}{ \left \lvert K \cap H_\alpha \right \rvert^{{1/(n-1)}}}.$$ 
    Thus,
    \begin{align*}
    \left \lvert K \cap H_\alpha \right \rvert \geq C_1(\alpha, n)\left \lvert K \cap \big(\xi^\perp + t_0\xi\big) \right \rvert \frac{\left \lvert K \cap \big(\xi^\perp + t_0\xi\big) \right \rvert^{{1/(n-1)}}}{\left \lvert K \cap H_\alpha \right \rvert^{{1/(n-1)}}},
    \end{align*}
    which implies
    \begin{align}
    \left \lvert K \cap H_\alpha \right \rvert \geq C_1(\alpha, n)^{\frac{n-1}{n}}\left \lvert K \cap \big(\xi^\perp + t_0\xi\big) \right \rvert. \label{eq12}
    \end{align}
    Now suppose $H_\alpha^* = H_\alpha^-$. Then the following inequality holds
    \begin{align*}
    &\left \lvert K \cap H_\alpha \right \rvert = \frac{\left \lvert N_2 \right \rvert n}{ {|\gamma - h|  }} \geq \frac{\left \lvert K \cap H_\alpha^- \right \rvert n}{ {|\gamma - h|  }}.
    \end{align*}
    By Theorem $4$ we have $\big((1 - C_2(\alpha, n)\big)|K| \leq |K \cap H_\alpha^-|$ and so the following inequalities hold
    \begin{align*}
    \vol{K \cap H_\alpha} \geq \frac{\left \lvert K \cap H_\alpha^- \right \rvert n}{ {|\gamma - h| }} &\geq \big(1 - C_2(\alpha, n)\big)\frac{\left \lvert K \right \rvert n}{ {|\gamma - h| }} \\
    &\geq \frac{1 - C_2(\alpha, n)}{C_2(\alpha, n)}\frac{\left \lvert K \cap H_\alpha^+ \right \rvert n}{ {|\gamma - h| }} \geq \frac{1 - C_2(\alpha, n)}{C_2(\alpha, n)}\frac{\left \lvert N_1 \setminus N_2 \right \rvert n}{ {|\gamma - h|}}.
    \end{align*}
    Expressing the volumes of $N_1$ and $N_2$ in terms of their bases, we get
    \begin{align*}
    |K \cap H_\alpha| \geq\ & \frac{1 - C_2(\alpha, n)}{C_2(\alpha, n)}\frac{(\left \lvert N_1 \right \rvert - \left \lvert N_2 \right \rvert) n}{ {|\gamma - h| }} \\
    =\ & \frac{1 - C_2(\alpha, n)}{C_2(\alpha, n)}\frac{\left \lvert K \cap \big(\xi^\perp + t_0\xi\big) \right \rvert  {|\gamma - t_0| }}{ {|\gamma - h| }} - \frac{1 - C_2(\alpha, n)}{C_2(\alpha, n)}\left \lvert K \cap H_\alpha \right \rvert.
    \end{align*}
    So,
    \begin{multline*}
       |K \cap H_\alpha| \geq \frac{\frac{1 - C_2(\alpha, n)}{C_2(\alpha, n)}}{1 + \frac{1 - C_2(\alpha, n)}{C_2(\alpha, n)}}\frac{|K \cap (\xi^\perp + t_0 \xi) ||\gamma - t_0|}{|\gamma - h|} \\
       = \left(1 - C_2(\alpha, n)\right)\frac{|K \cap (\xi^\perp + t_0 \xi) ||\gamma - t_0|}{|\gamma - h|}.
    \end{multline*}
    Again using the homothety of $N_1$ and $N_2$, we arrive at
    \begin{align*}
    \left \lvert K \cap H_\alpha \right \rvert& \geq \big(1 - C_2(\alpha, n)\big) \left \lvert K \cap \big(\xi^\perp + t_0\xi\big) \right \rvert \frac{\left \lvert K \cap \big(\xi^\perp + t_0\xi\big) \right \rvert^{{1/(n-1)}}}{\left \lvert K \cap H_\alpha \right \rvert^{{1/(n-1)}}},
    \end{align*}
    which implies
    \begin{align}
    \left \lvert K \cap H_\alpha \right \rvert \geq \big(1 - C_2(\alpha, n)\big)^{\frac{n-1}{n}} \left \lvert K \cap \big(\xi^\perp + t_0\xi\big)\right \rvert. \label{eq13}
    \end{align}
    
    Now to determine $D(\alpha, n)$ we need to find the minimum of the two constants in equations $(\ref{eq12})$ and $(\ref{eq13})$ for fixed $\alpha$. Note that $n\alpha \leq -\alpha$ for $\alpha \in (-1, 0]$. Then it follows that
    \begin{align*}
    \big(1 - C_2(\alpha, n)\big)^{\frac{n-1}{n}} = \left(\frac{n(\alpha+1)}{n+1}\right)^{n-1}\leq \left(\frac{n-\alpha}{n+1}\right)^{n-1}= C_1(\alpha, n)^{\frac{n-1}{n}},
    \end{align*}
    for all $\alpha \in (-1, 0]$, and thus we have our desired constant.
    
    We will now consider $\alpha \in (0, 1/n]$.   We claim that it is enough to solve the problem for the class of double cones and truncated cones as it was done   in Theorem $4$. Our plan of attack to prove this claim is to show that when we construct such  bodies following the procedure from Theorem $4$, we can only decrease $|K \cap H_\alpha|$ and only increase the volume of the maximal section of $K$. 
    It suffices to show this for the Schwarz symmetral $\calS_\xi K$ of $K$, which after abuse of notation we will denote by $K$. We will also employ Remark $3$ and prove the result for $\widebar K$ and $\widebar H_\alpha = \{x \in \RR^n : \inner{x}{\xi} = (\alpha + 1)\inner{g(\widebar K)}{\xi}\}$. Again after abuse of notation, we will write $K$ for $\widebar K$ and $H_\alpha$ for $\widebar H_\alpha$. Let $$t_0 = \min \{ t \in \RR: A_{K, \xi}(t) = \max_{t \in \RR} A_{K,\xi}(t) \},$$
    so that $K \cap (\xi^\perp + t_0\xi)$ is a section of $K$ orthogonal to $\xi$ of maximal volume. For brevity,  we will denote  $G_K=(\alpha + 1)\inner{g(K)}{\xi}$. We will split the analysis into two parts, according to whether $G_K < t_0 \le h_K(\xi)$ or $0 \le t_0 < G_K$. The case $t_0 =G_K$ is trivial.
    
    Suppose that $G_K < t_0 \le h_K(\xi)$. Then following the upper bound construction in Theorem $4$, we can construct a convex body $$L = \conv (r_1\ball{2}{n-1}, r_2\ball{2}{n-1}+ \mu \xi),$$ for some $r_1 \geq 0$ and $r_2 \geq 0$ such that $r_1 + r_2 > 0$, and $\mu$ such that $G_K < \mu \le h_K(\xi)$. Here, as before, $\ball{2}{n-1}$ stands for the unit Euclidean ball in $\xi^\perp$. Write
    $$r_{K, \xi}(t) = \omega^{-1/(n-1)}A^{1/(n-1)}_{K, \xi}(t) \quad \text{and} \quad r_{L, \xi}(t) = \omega^{-1/(n-1)} A^{1/(n-1)}_{L, \xi}(t),$$
   where $\omega = |\ball{2}{n-1}|$. Observe that $r_{L, \xi}$ is affine on its support, and    $r_{K, \xi}$ is concave on its support  by Lemma $1$. In fact we can write:
    $$r_{L, \xi}(t) = \frac{r_2 - r_1}{\mu}t + r_1.$$
    
    We claim that $r_2> r_1$, i.e.,  $r_{L, \xi}$ is increasing. To reach a contradiction, assume $r_{L, \xi}$ is  non-increasing. Since the graphs of $r_{L, \xi}$ and $r_{K, \xi}$ intersect at $t=G_K$ and at some other point $t<G_K$, the concavity of $r_{K, \xi}$ implies that $ r_{K, \xi}(t)\le r_{L, \xi}(t)$ for $t\ge G_K$. Since we are assuming that $G_K < t_0 $, we get 
    $r_{K, \xi}(t_0)\le r_{L, \xi}(t_0)\le r_{L, \xi}(G_K)=r_{K, \xi}(G_K)$.
    This means that  $t_0\le G_K$, which is a contradiction. 
    
    Let us denote $G_L= (\alpha + 1)\inner{g(L)}{\xi}$. Since $r_{L, \xi}$ is increasing, $\inner{g(L)}{\xi} \leq \inner{g(K)}{\xi}$ (as in the proof of Theorem $4$), and $A_{L, \xi}(G_K)=A_{K, \xi}(G_K)$, we obtain
    \begin{align*}
        \vol{L \cap \soa{x \in \RR^n : \inner{x}{\xi} = G_L}}& \le \vol{L \cap \soa{x \in \RR^n : \inner{x}{\xi} = G_K}} \\
       & = \vol{K \cap \soa{x \in \RR^n : \inner{x}{\xi} = G_K}}.
    \end{align*}
    We now want to show that the volume of the maximal section of $L$ is no smaller than the volume of the maximal section of $K$. Suppose the opposite, that is
    \begin{align}
        A_{K, \xi}(t_0) > A_{L, \xi}(\mu).\label{ineq}
    \end{align}
    Then it follows by the construction of $L$ and concavity of $r_{K, \xi}$ on its support that $\mu < t_0$. (Otherwise we would have $A_{K, \xi}(t_0) \le A_{L, \xi}(t_0) \le A_{L, \xi}(\mu))$. Raising both sides of \eqref{ineq} to the power $1/(n-1)$, we see \begin{align*}
        r_{K, \xi}(t_0) > r_{L, \xi}(\mu) = r_2.
    \end{align*}

   Observe that the linear functions $\frac{r_2 - r_1}{\mu}t + r_1$ and $\frac{r_2 - r_{K, \xi}(G_K)}{\mu - G_K}(t - G_K) + r_{K, \xi}(G_K)$ coincide. Indeed, at $t=\mu$ they are both equal to $r_2$, and at $t=G_K$ they are equal to $\frac{r_2 - r_1}{\mu}G_K + r_1= r_{L, \xi}(G_K)=r_{K, \xi}(G_K)$. Thus we obtain   
    \begin{align*}
        \vol{L \cap \{x \in \RR^n : \inner{x}{\xi} \geq G_K\}}& = \omega  \int_{G_K}^{\mu}\left(\frac{r_2 - r_1}{\mu}t + r_1\right)^{n-1} dt \\
        &=\omega  \int_{G_K}^{\mu}\left(\frac{r_2 - r_{K, \xi}(G_K)}{\mu - G_K}(t - G_K) + r_{K, \xi}(G_K)\right)^{n-1} dt\\
        & =   \frac{\omega}{n}\frac{\mu - G_K}{r_2 - r_{K, \xi}(G_K)} (r^n_2 - r^n_{K, \xi}(G_K)) \\
        &<   \frac{\omega}{n}\frac{t_0 - G_K}{r_2 - r_{K, \xi}(G_K)} (r^n_2 - r^n_{K, \xi}(G_K)) \\
        & = \omega  \int_{G_K}^{t_0}\left(\frac{r_2 - r_{K, \xi}(G_K)}{t_0 - G_K}(t - G_K) + r_{K, \xi}(G_K)\right)^{n-1} dt,
    \end{align*}
    where we used $\mu < t_0$ for the above inequality. Denote $$\zeta(t) =  \frac{r_2 - r_{K, \xi}(G_K)}{t_0 - G_K}(t - G_K) + r_{K, \xi}(G_K) .$$
    Note that $\zeta(G_K) = r_{K, \xi}(G_K)$ and $\zeta(t_0) =r_2$. Since by assumption $r_2 < r_{K, \xi}(t_0)$, it follows from concavity that $\zeta(t) < r_{K, \xi}(t)$ for all $t \in
    (G_K, t_0]$, and thus we have
    \begin{align*}
      \omega \int_{G_K}^{t_0}\zeta^{n-1}(t)\, dt < \int_{G_K}^{t_0}A_{K, \xi}(t)\, dt \le \vol{K \cap \{x \in \RR^n : \inner{x}{\xi} \geq G_K\}}. 
    \end{align*}
    Combining all of the above inequalities, we obtain
    \begin{align*}
        \vol{L \cap \soa{x \in \RR^n : \inner{x}{\xi} \geq G_K}} < \vol{K \cap \soa{x \in \RR^n : \inner{x}{\xi} \geq G_K}},
   \end{align*}
    reaching a contradiction to our construction in Theorem $4$. Therefore, we must have 
    \begin{align*}
        A_{K, \xi}\big(t_0) \leq A_{L, \xi}\big(\mu),
    \end{align*}
    as desired.

    From now on we can work with the body $L$ instead of $K$. We are interested in minimizing $\frac{A_{L, \xi}(G_L)}{A_{L,\xi}(\mu)}$.   
    As in Theorem $4$, computing this minimum explicitly is not feasible, but we will work around this fact. For now, it is enough to note that $A_{L, \xi}(t)$ is increasing in $t$ on its support, so it follows that $A_{L, \xi}(\inner{g(L)}{\xi}) \leq A_{L, \xi}(G_L)$. Hence we obtain for $\alpha \in (0, 1/n]$ the following inequalities
    \begin{align}
    \left(\frac{n}{n+1}\right)^{n-1} \leq \frac{A_{L, \xi}(\inner{g(L)}{\xi})}{A_{L, \xi}(\mu)} \le \frac{A_{L, \xi}(G_L)}{A_{L, \xi}(\mu)}  , \label{eq100}
    \end{align}
    where we used the result of Makai and Martini \eqref{eq4}.
    
    Now suppose that $0 \le t_0 < G_K$. Then following the lower bound construction in Theorem $4$, we can construct a convex body $M = \conv(0, r\ball{2}{n-1} + \beta\xi, v\xi)$ for some $v \geq h_K(\xi)$, $r > 0$, and $\beta$ such that $0 \le \beta \le G_K$. Note that $M \cap (\xi^\perp + \beta\xi)$ is the maximal section of $M$ in the direction $\xi$.     
    To be precise, in Theorem 4 the  construction that produced a double cone was performed for $\alpha \in (0, 1/n)$. However the same construction works also for $\alpha=1/n$, unless $H_{1/n}$ is a supporting hyperplane to $K$. In the latter case we can see that $K$ is a body that can be reduced to the cone   \eqref{cone} and therefore is not an extremizer. Thus below we will exclude such bodies and work with all $\alpha \in (0, 1/n]$. In particular, we will have $\beta < v$.

   As was done above above, we may write: $$r_{M, \xi}(t) = \omega^{-1/(n-1)} A_{M, \xi}^{1/(n-1)}(t) = \begin{cases} 
            \frac{r}{\beta}t, & \mbox{ if } t \in [0, \beta], \\
            \frac{r}{\beta - v}(t - \beta) + r,   & \mbox{ if } t \in (\beta, v].
           \end{cases}$$
    Recall from Theorem $4$ that  $\inner{g(M)}{\xi} \geq \inner{g(K)}{\xi}$. Denoting $G_M=(\alpha + 1)\inner{g(M)}{\xi}$, we see that $\beta\le G_K\le G_M$, and hence since $A_{M, \xi}$ is decreasing on $[\beta,v]$, it follows that 
    \begin{align*}
        \vol{M \cap \soa{x \in \RR^n : \inner{x}{\xi} = G_M}} & \le  \vol{M \cap \soa{x \in \RR^n : \inner{x}{\xi} = G_K}} \\
        &= \vol{K \cap \soa{x \in \RR^n : \inner{x}{\xi} = G_K}}.
    \end{align*}
    Now we want to show that the volume of the maximal section of $M$ is no smaller than the volume of the maximal section of $K$. Again, suppose the opposite, that is
   \begin{equation}A_{K, \xi}(t_0) > A_{M, \xi}(\beta).\label{ineq1}
   \end{equation}
    Then it follows by the construction of $M$ and concavity of $r_{M, \xi}$ on its support that $\beta > t_0$. (Otherwise, using that $ A_{K, \xi}(t) \le  A_{M, \xi}(t)$ for $t\in[\beta,G_K]$, $A_{M, \xi}(t)$ is decreasing for $t\ge \beta$, and $\beta\le t_0 \le G_K$, we would get $ A_{K, \xi}(t_0) \le  A_{M, \xi}(t_0)\le A_{M, \xi}(\beta) $). Raising both sides of \eqref{ineq1} to the power $1/(n-1)$, we again obtain
    $$r_{K, \xi}(t_0) > r_{M, \xi}(\beta) = r.$$
    
    Since $r_{M, \xi}(G_K) = r_{K, \xi}(G_K)$, we have $\frac{r}{\beta - v} = \frac{r - r_{K, \xi}(G_K)}{\beta - G_K}$, and therefore
    \begin{align*}
      \frac{1}{\omega}  |M \cap &\{x \in \RR^n : \inner{x}{\xi} \leq G_K\}|\\
        & = \int_{0}^{\beta}\left(\frac{r}{\beta}\right)^{n-1} t^{n-1} dt  + \int_{\beta}^{G_K}\left(\frac{r}{\beta - v}(t - \beta) + r\right)^{n-1} dt \\
        &= \int_{0}^{\beta}\left(\frac{r}{\beta}\right)^{n-1} t^{n-1} dt +\int_{\beta}^{G_K}\left(\frac{r - r_{K, \xi}(G_K)}{\beta - G_K}(t - \beta) + r\right)^{n-1} dt\\
       &= \frac{1}{n}\left( \beta r^{n-1}  + \frac{ \beta - G_K}{r- r_{K, \xi}(G_K) } \left(r^n_{K, \xi}(G_K) - r^n\right) \right).
          \end{align*}
     Observe that the latter linear function of $\beta$ is decreasing, since its slope is negative:
          $$r^{n-1}+ \frac{r^n_{K, \xi}(G_K) - r^n}{r- r_{K, \xi}(G_K) }= \frac{r_{K, \xi}(G_K)  \left(r^{n-1}_{K, \xi}(G_K)  -r^{n-1}\right)}{r- r_{K, \xi}(G_K) }<0.$$
          Therefore, using    $\beta > t_0$, we obtain          
          \begin{align*}
            \frac{1}{\omega}  |M \cap &\{x \in \RR^n : \inner{x}{\xi} \leq G_K\}|\\
       &< \frac{1}{n}\left( t_0 r^{n-1}  + \frac{ t_0 - G_K}{r- r_{K, \xi}(G_K) } \left(r^n_{K, \xi}(G_K) - r^n\right) \right) \\
       &= \int_{0}^{t_0}\left(\frac{r}{t_0}\right)^{n-1} t^{n-1} dt +\int_{t_0}^{G_K}\left(\frac{r - r_{K, \xi}(G_K)}{t_0 - G_K}(t - t_0) + r\right)^{n-1} dt, 
    \end{align*}
   Let us write 
    $$\zeta(t) = \begin{cases} 
            \frac{r}{t_0}t, & \mbox{ if } t \in [0, t_0], \\
            \frac{r - r_{K, \xi}(G_K)}{t_0 - G_K}(t - t_0) + r, & \mbox{ if }t \in (t_0, G_K].
           \end{cases}$$
    Note that $\zeta(G_K) =  r_{K, \xi}(G_K)$. Since by assumption $r < r_{K, \xi}(t_0)$, it follows from concavity that $\zeta(t) < r_{K, \xi}(t)$ for all $t \in [0, G_K)$, and thus we have
    \begin{align*}
      \omega \int_{0}^{G_K}\zeta^{n-1}(t)\, dt < \int_{0}^{G_K}A_{K, \xi}(t)\, dt = \vol{K \cap \{x \in \RR^n : \inner{x}{\xi} \leq G_K\}}. 
    \end{align*}
    Combining all of the above inequalities, we obtain
    \begin{align*}
        \vol{M \cap \soa{x \in \RR^n : \inner{x}{\xi} \leq G_K}} < \vol{K \cap \soa{x \in \RR^n : \inner{x}{\xi} \leq G_K}},
   \end{align*}
    reaching a contradiction to our construction in Theorem $4$. Therefore, we must have 
    \begin{align*}
        A_{K, \xi}\big(t_0) \leq A_{M, \xi}\big(\beta),
    \end{align*}
    and hence, it suffices to work with $M$ instead of $K$. After rescaling, we may assume that $h_{M}(\xi) = 1$ and $\vol{r \ball{2}{n-1}} = n.$ As before, we will define
    \begin{align*}
    M_2 = M \cap \{x \in \mathbb{R}^n \mid \langle x,\xi  \rangle \geq \beta\}.
\end{align*}
Taking our computations from Theorem $4$, we have $\vol{M_2} = 1 - \beta$ and
    \begin{align*}
    \left\lvert M \cap H_\alpha^+ \right\rvert = \frac{(1-G_M)^n}{(1-\beta)^{n-1}},
    \end{align*}
     where $$G_M   = (\alpha + 1)\frac{\beta(n-1) + 1}{n+1}.$$
    Expressing the volumes of the sections we are interested in terms of the volumes of the cones $M_2$ and $M \cap H_\alpha^+$, we can write
    \begin{equation*}
    \frac{\vol{M \cap \{x \in \RR^n : \inner{x}{\xi} = G_M\}}}{\bvol{M \cap (\xi^\perp + \beta\xi)}} = \frac{|M \cap H_\alpha^+|}{|M_2|}\frac{1-\beta}{1-G_M}  
    = \left(\frac{1 - G_M}{1 - \beta}\right)^{n-1}.
    \end{equation*}
    Denote by $\psi$ the above function of $\beta$. Our goal is to find the minimum of $\psi$ when $0\le \beta\le G_M$, or equivalently when  $\beta \in [0,\frac{\alpha + 1}{2-(n-1)\alpha } ]$ (as taken from Theorem $4$). One can check that $\psi$ is increasing in $\beta$ and therefore the minimum of $\psi$ is attained at $\beta=0$,
    \begin{align}
    \psi(\beta)\ge \psi(0)= \left(\frac{n-\alpha}{n+1}\right)^{n-1}. \label{eq110}
    \end{align}
  Additionally, since $\beta=0$ the extremizing body is a cone.

    Now to determine the value of $D(\alpha, n)$ for $\alpha \in (0, 1/n]$, we need to find the lower of the two constants in $(\ref{eq100})$ and $(\ref{eq110})$. It is enough to note for $\alpha \in (0, 1/n]$ that $$\left(\frac{n-\alpha}{n+1}\right)^{n-1} \leq \left(\frac{n}{n+1}\right)^{n-1},$$
    and the result follows.
    
    Discussing equality cases, we see for $\alpha \in (-1, 0]$ that equality follows from the equality cases for the upper bound in Theorem $3$ (which comes from Gr\"unbaum's original theorem), and thus the equality bodies are, up to translation, cones of the form $  \conv(B, v)$ with $B$ an $(n-1)$-dimensional convex body lying parallel to $\xi^\perp$ in $\xi^+$ and vertex $v$ lying in $\xi^-$. For $\alpha \in (0, 1/n]$, the equality bodies are cones of the form $\conv(B, v)$ with $B$ an $(n - 1)$-dimensional convex body lying parallel to $\xi^\perp$ in $\xi^-$ and vertex $v$ lying in $\xi^+$; see Figure \ref{fig5}. For $\alpha \in (1/n, n)$, there are many bodies for which $K\cap H_\alpha = 0$. Therefore we omit this case from Figure \ref{fig5}. 
    \begin{figure}[h]
    \centering
    \includegraphics[scale=0.23]{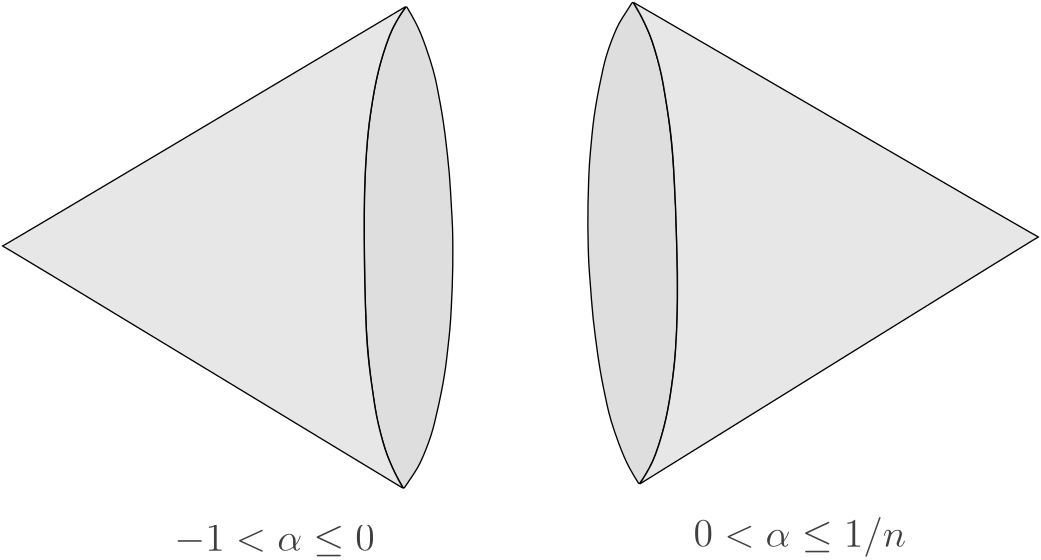}
    \caption{Extremizing shapes for Theorem $5$. }\label{fig5}
    \end{figure}
    \end{proof}

\end{document}